\documentclass[12pt,twoside]{amsart}
\usepackage{amssymb}
\usepackage{a4wide}
\usepackage[latin1]{inputenc}
\usepackage[T1]{fontenc}
\usepackage{times}

\def\hpq0{h^{p,q}_{\leq 0}}
\def\Hpq0{\H_{\leq 0}^{p,q}}

\def\dbar{\bar\partial}
\def\ddbar{\partial\dbar}

\def\C{{\mathbb C}}

\def\H{{\mathcal H}}
\def\E{{\mathcal E}}

\def\Pop{{\mathbb P}}

\def\Re{{\rm Re\,  }}

\def\be{\begin{equation}}
\def\ee{\end{equation}}

\newtheorem{thm}{Theorem}[section]
\newtheorem{lma}[thm]{Lemma}
\newtheorem{cor}[thm]{Corollary}
\newtheorem{prop}[thm]{Proposition}

\theoremstyle{definition}

\theoremstyle{remark}

\newtheorem{preremark}{Remark}
\newtheorem{preex}{Example}

\numberwithin{equation}{section}

\title[]
{The openness conjecture for projective manifolds}

\address{Department of Mathematics\\Chalmers University
  of Technology \\
 } 

\email{ bob@chalmers.se}

\author[]{ Bo Berndtsson}

\begin{document}

\begin{abstract}We give a proof of the openness conjecture of Demailly and Koll\'ar for positively curved singular metrics on ample line bundles over projective varieties. As a corollary it follows that the openness conjecture for plurisubharmonic functions with isolated sigularities holds. 
\end{abstract}
\maketitle

\section{Introduction}
Let $ u$ be a plurisubharmonic function defined in a neighbourhood of the origin of $\C^n$ such that $e^{-u}$ lies in $L^1$. The {\it openness conjecture}, first proposed by Demailly and Koll\'ar in \cite{Demailly-Kollar}, says that then there is a number $p>1$ such that $e^{-u}$ lies in $L^p$, possibly after shrinking the neighbourhood. This conjecture has attracted a good deal of attention; in particular it has been completely proved in dimension 2 by Favre and Jonsson, \cite{Favre-Jonsson}. 
In arbitrary dimension it is still open (no pun intended), but has been reduced to a purely algebraic statement in \cite{Jonsson-Mustata}.

In this paper we will  prove a global version of the openness conjecture for metrics on line bundles over projective manifolds.
\begin{thm} Let $X$ be a projective manifold and let $L$ be a positive line bundle over $X$. Let $\phi$ be a possibly singular metric of nonnegative curvature on $L$, and let $\phi_0$ be a smooth positively curved metric on $L$. Assume that 
$$
\int_X e^{-(\phi-\phi_0)} d\mu <\infty
$$
where $\mu$ is some smooth volume form on $X$. Then there is a number $p>1$ such that 
$$
\int_X e^{-p(\phi-\phi_0)} d\mu <\infty.
$$
\end{thm}
In particular the theorem applies to projective space, $\Pop^n$.
By a simple max-construction one can show that any local plurisubharmonic function with an isolated singularity ( by this we mean  that $\phi(0)=-\infty$ and $\phi$ is bounded for $|z|>1/2$) can be extended to a metric with positive curvature on some $\mathcal{O}(k)$ over $\Pop^n$, with no additional singularities. 
As a consequence we see that the openness conjecture holds in any dimension for functions with isolated singularities.

\begin{cor} Let $u$ be a plurisubharmonic function in the unit ball of $\C^n$, $B$, with an isolated singularity at 0. Assume that 
$$
\int_B e^{-u}<\infty.
$$
Then there is a number $p>1$ such that
$$
\int_B e^{-pu}<\infty.
$$
\end{cor}

The proof of the theorem is inspired by a result from \cite{Berman-Berndtsson}. There we proved that the Schwarz symmetrization, $u^*$,  of an $S^1$-invariant plurisubharmonic function in the ball,$u$,   is again plurisubharmonic. Here $S^1$-invariance means that $u(e^{i\theta}z)=u(z)$ for any $e^{i\theta}$ on $S^1$. Since the Schwarz symmetrization of $u$ is equidistributed with $u$, it holds that
$$
\int_B F(u^*)=\int_B F(u)
$$
for any (measurable) function $F$. Choosing $F(t)=e^{-t}$ and $F(t)=e^{-\epsilon t}$,  this  reduces the openness problem for $S^1$-invariant functions to the case of radial functions. As a consequence we get
\begin{prop}
The openness conjecture holds for any $S^1$-invariant plurisubharmonic function in the unit ball.
\end{prop}
The proof of Proposition 1.3, via the result from \cite{Berman-Berndtsson}, depends ultimately on a complex variant of the Brunn-Minkowski inequality from \cite{Berndtsson}. The argument can be rephrased in the following form. A consequence of the 'Brunn-Minkowski'-inequality is that
if $u$ is plurisubharmonic and $S^1$-invariant, the volume of the sublevel sets
$$
\Omega(s):=\{z; u(z)<-s\}
$$
is a logconcave function of $s$. The integral of $e^{-u}$ can be written as
\be
\int_0^\infty e^s |\Omega(s)|ds + \omega_n,
\ee
with $\omega_n$ the volume of the unit ball, 
and the logconcavity is the key to studying the convergence of this integral.
In fact, if $|\Omega(s)|$ is logconcave, the integral (1.1) converges (if and) only if $|\Omega(s)|$ decreases like  $e^{-(1+\epsilon)s}$ at infinity (cf Theorem 3.1).

In the situation of Theorem 1.1, instead of looking at just volumes of sets or integrals of functions, we look at the $L^2$-norms on the space $H^0(X, K_X+kL)$ induced by our metric $\phi$ in section 2. We then find a representation of such an $L^2$-norm as an integral over $(0,\infty)$ of weaker norms depending on the variable $s$. These weaker norms have a property analogous to logconcavity - they define an hermitean metric on a certain vector bundle of positive curvature. This positivity property is finally shown in the last section to imply Theorem 1.1.

\section{Hermitean norms on $H^0(K_X+kL)$}

Let $X$ be a projective manifold and let $L$ be a positive line bundle over $X$. We will consider a possibly singular metric, $\phi$ with $i\ddbar\phi\geq 0$, and we also let $\phi_0$ be a smooth positively curved reference metric on $L$. For $\sigma$, an element in $H^0(K_X+kL)$ we define the $L^2$-norm
\be
\|\sigma\|^2:= c_n\int_X \sigma\wedge\bar\sigma e^{-\phi-(k-1)\phi_0}=
c_n\int_X \sigma\wedge\bar\sigma e^{-(\phi-\phi_0)-k\phi_0}.
\ee
Clearly, this norm is finite for any $\sigma$ in $H^0(K_X+kL)$ if and only if $e^{-\phi}$ is locally integrable, provided that $k$ has been chosen so large that $K_X+kL$ is base point free. For $s>0$ we define a regularization of $\phi$ by
\be
\phi_s:=\max(\phi+s, \phi_0).
\ee
If we normalize so that $\phi\leq\phi_0$, we get that $\phi_s=\phi_0$ for $s=0$ so there is no conflict in notation. 
With $\phi_s$ we  associate the norms
\be
\|\sigma\|^2_s:= c_n\int_X\sigma\wedge\bar\sigma e^{-2\phi_s-(k-2)\phi_0}=
c_n\int_X \sigma\wedge\bar\sigma e^{-2(\phi_s-\phi_0)-k\phi_0}.
\ee
(Notice the factor 2 in front of $(\phi_s-\phi_0)$ as opposed to 1 in the formula in (2.1).)
\begin{prop}
If $\sigma$ is an element of $H^0(K_X+kL)$,
$$
2\|\sigma\|^2=\int_0^\infty e^s \|\sigma\|^2_s ds +\|\sigma\|^2_0.
$$
\end{prop}
For the proof we use the following lemma.
\begin{lma} If $x<0$
$$
\int_0^\infty e^s e^{-2\max(x+s,0)}ds +1= 2e^{-x}.
$$
More generally, if $0<p<2$,
$$
\int_0^\infty e^{ps} e^{-2\max(x+s,0)}ds +1/p= C_p e^{-px}.
$$
\end{lma}
\begin{proof}
$$
\int_0^\infty e^s e^{-2\max(x+s,0)}ds =\int_0^{-x} e^s ds +
 e^{-2x}\int_{-x}^{\infty} e^{-s}ds = 2e^{-x}-1.
$$
This proves the first part; the second part is of course also proved by direct computation.
\end{proof}
\noindent To prove the proposition we use
$$
\|\sigma\|^2_s=c_n\int_X \sigma \wedge\bar\sigma e^{-2(\phi_s-\phi_0)-k\phi_0}.
$$
Note that $\phi_s-\phi_0=\max(\phi-\phi_0 +s,0)$. By the lemma
$$
\int_0^\infty e^s \|\sigma\|^2_s ds= 2c_n\int_X\sigma\wedge\bar\sigma e^{-(\phi-\phi_0)-k\phi_0} -c_n\int_X\sigma\wedge\bar\sigma e^{-k\phi_0}.
$$
This completes the proof. \qed

\bigskip

For later reference we note that by the last part of the lemma,  if $0<p<2$,
\be
\int_0^\infty e^{ps} \|\sigma\|^2_s ds= C_p c_n\int_X \sigma\wedge\bar\sigma e^{-p(\phi_s-\phi_0)-k\phi_0}-(1/p)\|\sigma\|^2_0.
\ee

\bigskip

We finally quote a particular case of a result from \cite{2Berndtsson} that is the most important ingredient in the proof of Theorem 1.1. We let $D$ be a domain in $\C$ and let $E:=H^0(X, K_X+F)$, where $F$ is a positive line bundle over $X$. Denote by $\E:=D\times E$, the trivial vector bundle with fiber $E$ over $D$. Let for $\zeta$ in $D$, $\psi_\zeta$ be a metric on $F$.
\begin{thm} With notation as above, define an hermitean metric on $\E$ by
\be
\|\sigma\|^2_\zeta:=c_n\int_X\sigma\wedge\bar\sigma e^{-\psi_\zeta}.
\ee
Assume that $i\ddbar_{\zeta,X}\psi_\zeta\geq 0$, i e that $\psi_\zeta$ is plurisubharmonic on $D\times X$. Then the curvature of the metric (2.5) on $\E$ is nonnegative.
\end{thm}
This result applies in particular to our present setting, with $D$ equal to the right half plane and $F=kL$. Then
$$
\psi_\zeta= 2\phi_s+(k-2)\phi_0
$$
for $s=\Re\zeta$. By the definition (2.2), $\phi_{\Re\zeta}$ and hence $\psi_\zeta$  are  plurisubharmonic on all of $D\times X$. Hence, by Theorem 2.3,  the norms $\|\sigma\|_s$ on $H^0(K_X+kL)$ define a metric on $\E$ of positive curvature. Moreover, they depend only on $s=\Re\zeta$. It is easily checked that if  all $\|\sigma\|_s$ can be simultanously diagonalised in some fixed basis, with diagonal entries $\omega_j(s)$, then the positivity of the curvature means that all $\omega_j(s)$ are logconcave. This is  our substitute for the logconcavity of $|\Omega(s)|$, mentioned in the introduction. 

\section{Integrals of quadratic forms and the proof of Theorem 1.1}
We continue the discussion from the previous section and specialize to $D=U$, the right half plane. 

\begin{thm}
Let $\|\cdot\|_s$ be a family of Hilbert norms on some finite dimensional vector space $E$ such that the induced hermitean metric on the trivial vector bundle $\E:=U\times E$, $\|\cdot\|_{\Re \zeta}$, has positive curvature over $U$. Then the integrals
$$
\int_0^{\infty} e^s \|\sigma\|^2_s ds
$$
converge for all $\sigma$ in $E$ if and only if there are $\epsilon>0$ and $s_0$ such that
\be
 \|\sigma\|^2_s\leq e^{-(1+\epsilon)s} \|\sigma\|^2_0
\ee
for $s>s_0$ and any $\sigma$ in $E$.\end{thm}
\begin{proof}
One direction is of course clear; if (3.1) holds the integral converges. So, assume that (3.1) does not hold. Then, for any $\epsilon>0$ we can find  $s>1/\epsilon$ and some $\sigma$ in $E$ such that
\be
 \|\sigma\|^2_s > e^{-(1+\epsilon)s} \|\sigma\|^2_0.
\ee
By the spectral theorem we can choose an orthonormal basis, $e_j$, for $\|\cdot\|_0$ that diagonalises  $\|\cdot\|_s$. If 
$$
\sigma=\sum c_j e_j
$$
we can write
$$
\|\sigma\|^2_0=\sum |c_j|^2 \quad \text{and}\quad \|\sigma\|^2_s=\sum |c_j|^2e^{s\lambda_j}
$$
since the eigenvalues are positive. By (3.2), at least one $\lambda_j$ - say $\lambda_0$ - is larger than $-(1+\epsilon)$. 

\bigskip

We now define another family of norms $|\cdot|_t$ for $0\leq t\leq s$ by
$$
|\sigma|^2_t=\sum |c_j|^2 e^{t\lambda_j}
$$
Since $e_j(\zeta)=e_j e^{-\lambda_j\zeta/2}$ defines a global holomorphic orthonormal frame,  $|\sigma|^2_{\Re \zeta}$ defines a hermitean metric on $\E$ of zero curvature. Moreover the new norms agree with the previous ones for $t=0$ and $t=s$. By the maximum principle for positive metrics (see e g \cite{Berman-Keller}, Lemma 8.11) we have
$$
\|\sigma\|^2_t\geq |\sigma|^2_t
$$
for $0\leq t\leq s$. Choose $\sigma=e_0$. Then we conclude that
$$
\|e_0\|^2_t\geq |e_0|^2_t\geq \|e_0\|^2_0 e^{-(1+\epsilon)t}.
$$
Therefore
$$
\int_0^s e^t\|e_0\|^2_t dt\geq \|e_0\|^2_0\int_0^s e^{-\epsilon t} dt\geq
 \|e_0\|^2_0 (1-e^{-1})/\epsilon,
$$
since $s>1/\epsilon$. Since $\epsilon$ can be taken arbitrarily small we see that there is no constant such that
$$
\int_0^{\infty} e^s \|\sigma\|^2_s ds\leq C\|\sigma\|_0^2
$$
for all $\sigma$ in $E$. 
Since all norms on a finite dimensional vector space are equivalent,  the integral in the right hand side cannot converge, which completes the proof. 

\end{proof}

\bigskip

We can now combine Proposition 2.1 and Theorem 3.1 to prove Theorem 1.1.
The hypothesis implies that
$$
\|\sigma\|^2=c_n\int_X \sigma\wedge\bar\sigma e^{-(\phi-\phi_0)} e^{-k\phi_0}
$$
is finite for any $\sigma$ in $H^0(K_X+kL)$. By Proposition 2.1 this implies that the integrals
$$
\int_0^\infty e^s \|\sigma\|^2_s ds
$$
converge for any $\sigma$ in $H^0(K_X+kL)$. By Theorem 3.1 there are $\epsilon>0$ and $s_0$ such that 
$$
\|\sigma\|^2_s\leq e^{-(1+\epsilon)s}\|\sigma\|^2_0
$$
if $s>s_0$. Hence there is some $p>1$ such that 
$$
\int_0^\infty e^{ps} \|\sigma\|^2_s ds
$$
converges as well. By  (2.4) we then have that
$$
c_n\int_X \sigma\wedge\bar\sigma e^{-p(\phi-\phi_0)} e^{-k\phi_0}
$$
is finite. If $k$ is so large that $K_X+kL$ is base point free this gives that
$$
\int_X e^{-p(\phi-\phi_0)} d\mu <\infty,
$$
so we are done.

\def\listing#1#2#3{{\sc #1}:\ {\it #2}, \ #3.}


\begin{thebibliography}{9999}



\bibitem{Berman-Berndtsson}\listing{Berman R, Berndtsson, B}{Symmetrization of plurisubharmonic and convex functions}{arXiv:1204/0931}

\bibitem{Berman-Keller}\listing{Berman R, Keller, J}{Bergman geodesics}{in Complex Monge-Ampere equations and geodesics in the space of Kahler metrics, Springer Lecture Notes in Math 2038}


\bibitem{Berndtsson}\listing{Berndtsson B}{Subharmonicity of the Bergman kernel and some other functions associated to pseudoconvex domains}{Ann Inst Fourier, 56 (2006) pp 1633-1662}

\bibitem{2Berndtsson}\listing{Berndtsson, B}{Curvature of vector bundles
    associated to holomorphic fibrations  }{Ann Math 169 2009, pp
    531-560 }
\bibitem{Demailly-Kollar}\listing{Demailly, J-P and Koll\'ar, J}{Semicontinuity of complex singularity exponents and K\"ahler-Einstein metrics on Fano orbifolds}{Ann Sci Ecole Norm Sup 34 (2001) pp 525-556}



\bibitem{Favre-Jonsson}\listing{Favre, C and Jonsson, M}{Valuations and multiplier ideals}{J Amer Math Soc 18 (2005) pp 655-684}

\bibitem{Jonsson-Mustata}\listing{Jonsson, M and Mustata, M}{An algebraic approach to the openness conjecture of Demailly and Koll\'ar}{ArXiv:1205/4273}


\end{thebibliography}
\end{document}